\documentclass[11pt,onecolumn]{article}

\usepackage[a4paper,margin=1.15in]{geometry}
\usepackage{lmodern}
\usepackage{microtype}
\usepackage[dvipsnames]{xcolor}
\usepackage{amsmath}
\usepackage{amssymb}
\usepackage{amsthm}
\usepackage{aliascnt}
\usepackage{graphicx}
\usepackage{mathtools}
\usepackage{authblk}
\usepackage{enumitem}
\usepackage{titlesec}
\usepackage{fancyhdr}
\usepackage{hyperref}
\usepackage{cleveref}

\definecolor{linkblue}{HTML}{1F4E79}
\hypersetup{
    colorlinks=true,
    linkcolor=linkblue,
    citecolor=linkblue,
    urlcolor=linkblue
}

\titleformat{\section}{\large\bfseries}{\thesection}{0.75em}{}
\titleformat{\subsection}{\normalsize\bfseries}{\thesubsection}{0.75em}{}
\titlespacing*{\section}{0pt}{2.3ex plus 0.8ex minus 0.2ex}{1.1ex plus 0.2ex}
\titlespacing*{\subsection}{0pt}{1.8ex plus 0.6ex minus 0.2ex}{0.8ex plus 0.2ex}

\setlist[itemize]{leftmargin=2.2em,itemsep=0.25ex,topsep=0.6ex}
\usepackage{tikz-cd}

\theoremstyle{plain}
\newtheorem{theorem}{Theorem}[section]

\newaliascnt{corollary}{theorem}
\newtheorem{corollary}[corollary]{Corollary}
\aliascntresetthe{corollary}

\newaliascnt{proposition}{theorem}
\newtheorem{proposition}[proposition]{Proposition}
\aliascntresetthe{proposition}

\newaliascnt{lemma}{theorem}

\aliascntresetthe{lemma}

\theoremstyle{definition}
\newaliascnt{definition}{theorem}
\newtheorem{definition}[definition]{Definition}
\aliascntresetthe{definition}

\newaliascnt{example}{theorem}
\newtheorem{example}[example]{Example}
\aliascntresetthe{example}

\newaliascnt{remark}{theorem}
\newtheorem{remark}[remark]{Remark}
\aliascntresetthe{remark}

\crefname{theorem}{theorem}{theorems}
\Crefname{theorem}{Theorem}{Theorems}
\crefname{corollary}{corollary}{corollaries}
\Crefname{corollary}{Corollary}{Corollaries}
\crefname{proposition}{proposition}{propositions}
\Crefname{proposition}{Proposition}{Propositions}
\crefname{lemma}{lemma}{lemmas}
\Crefname{lemma}{Lemma}{Lemmas}
\crefname{definition}{definition}{definitions}
\Crefname{definition}{Definition}{Definitions}
\crefname{example}{example}{examples}
\Crefname{example}{Example}{Examples}
\crefname{remark}{remark}{remarks}
\Crefname{remark}{Remark}{Remarks}

\newcommand{\C}{\mathbb{C}}
\newcommand{\R}{\mathbb{R}}
\newcommand{\N}{\mathbb{N}}

\newcommand{\littletaller}{\mathchoice{\vphantom{\big|}}{}{}{}}
\newcommand\restr[2]{{
		\left.\kern-\nulldelimiterspace
		#1 
		\littletaller 
		\right|_{#2} 
}}

\title{On Subhomogeneous Operator Systems}
\author{Markus Dannem\"uller}
\author{Tim Netzer}
\affil{\small Department of Mathematics, University of Innsbruck, Austria}
\date{\today}

\begin{document}
\maketitle
\begin{abstract}
    We study subhomogeneity for finite-dimensional operator systems, and  collect and extend characterizations in terms of the $C^*$-envelope, $d$-maximality, complete positivity, dual $d$-minimality, and non-commutative boundary conditions. We then show that the dual of a subhomogeneous operator system, while not necessarily subhomogeneous itself,  is always a quotient of a subhomogeneous system. We complement these characterizations with examples and counterexamples, including minimal and maximal systems over certain polyhedral cones.
\end{abstract}

\section{Introduction}

Operator systems provide a framework for studying the self-adjoint parts of unital $C^*$-algebras and the completely positive maps between them.  In finite dimensions, this viewpoint connects several themes: matrix orders, matrix convexity, and the geometry of spectrahedra.  A central problem is to understand when the higher matrix levels of an operator system are already controlled by finitely many low levels.  This question is natural both from the perspective of operator algebra, where it is related to finite-dimensional and subhomogeneous representations, and from the perspective of convexity, where it asks how much of a matrix convex set is determined by bounded-size matrix points.

For $C^*$-algebras, subhomogeneity is a repesentation theoretic finiteness condition: a $C^*$-algebra is $d$-subhomogeneous if all irreducible representations have dimension at most $d$.  Equivalently, such an algebra embeds into matrices of size $d$ over a commutative $C^*$-algebra, and it is characterized by automatic complete positivity of $d$-positive maps.  In this paper we study the corresponding notion for finite-dimensional operator systems, which has recently been introduced \cite{kiri2024}. In our approach, the relevant replacement for an embedding of a $C^*$-algebra is a realization of the operator system inside ${\rm Mat}_d(A)$ with $A$ commutative; equivalently, one may ask whether the $C^*$-envelope is $d$-subhomogeneous. 

Our first main result is a collection and extension of equivalent characterizations of $d$-subhomogeneity.  For a finite-dimensional operator system $S$, the following conditions are all equivalent: $S$ admits a $d$-commutative realization, its $C^*$-envelope is $d$-subhomogeneous, $S$ is $d$-maximal, every $d$-positive map into $S$ is completely positive, the dual system $S^\vee$ is $d$-minimal, and a certain boundary-compression condition holds at all higher matrix levels.  These equivalences unify algebraic, order-theoretic, duality-theoretic, and geometric viewpoints. It also shows that our notion coincides with the one from \cite{kiri2024}.

We then study the behavior of duals.  Duality exchanges maximality and minimality, so one should not expect the dual of a subhomogeneous operator system to remain subhomogeneous.  Nevertheless, we prove a positive substitute: the dual of a $d$-subhomogeneous operator system is always a quotient of a $d$-subhomogeneous system.  This complements related projection results for finite-dimensional realizations of matrix convex sets, such as those in \cite{heltonklepmccullough}.

The final part of the paper develops examples.  For cones over cross-polytopes we show a sharp contrast: the interval and square cases have subhomogeneous  minimal operator systems, whereas dimension at least three gives minimal systems that are not subhomogeneous. 

These examples show that even for polyhedral cones, subhomogeneity depends delicately on the dimension and geometry of the cone and on the matrix convex structure it induces.

Throughout the paper, all operator systems are finite-dimensional. This ensures, among other things, that duals always exist. We expect, however, that many of the results can also be adapted to the infinite-dimensional setting.

\section{Preliminaries on Operator Systems}

We begin by recalling the basic notions and notation for operator systems and \(C^*\)-algebras, that will be used throughout. See  \cite{pau} for more details and proofs, for instance.

Throughout, all general vector spaces are finite-dimensional real vector spaces. In the operator algebra setting, one often works with complex vector spaces. However, these always carry an involution, all maps are $*$-linear, and positivity is determined by the self-adjoint part of the space. Thus everything can be reduced to the real case, and no generality is lost by this restriction.

A subset \(C \subseteq V\) of a vector space \(V\) is a \emph{convex cone} if \(C + C \subseteq C\) and \(\R_{\geqslant 0} \cdot C \subseteq C\). It is \emph{proper} if it is closed, pointed, and has non-empty interior, i.e.\ if \(C \cap -C = \lbrace 0 \rbrace\) and \(\operatorname{int}(C) \neq \emptyset\).
Denote by \({\rm Her}_s(\C) \subseteq {\rm Mat}_s(\C)\) the real vector space of hermitian matrices of size \(s\). 

An \emph{operator system} \(S\) on \(V\) is a collection of proper convex cones 
\[
    S_s \subseteq V \otimes {\rm Her}_s(\C)\eqqcolon {\rm Her}_s(V)
\] for \(s \in \N\), together with a distinguished interior point \(u \in S_1\), called the \emph{order unit}. The collection \((S_s)_{s \in \N}\), also called a \emph{matrix order}, is required to be closed under conjugation by scalar matrices: for all \(s, t \in \N\) and \(m \in {\rm Mat}_{s,t}(\C)\), we have \(m^* S_s m \subseteq S_t\).

Every $C^*$-algebra $A$ gives rise to an abstract operator system structure on the self-adjoint space $A_{\rm sa}$, by taking the cone of positive elements inside the self-adjoint part of $A\otimes{\rm Mat}_s(\C)={\rm Mat}_s(A)$ as the cone at level $s$. By slight abuse of notation, we denote this operator system structure again by $A$.

Given operator systems \(S\) and \(T\) on vector spaces \(V\) and \(W\), a linear map \(\phi \colon V \to W\) is called \emph{positive} if \(\phi(S_1) \subseteq T_1\). It is called \emph{\(s\)-positive} if \((\phi \otimes {\rm id}_{{\rm Her}_s(\C)})(S_s) \subseteq T_s\), and \emph{completely positive (cp)} if it is \(s\)-positive for every \(s \in \N\). If \(u_S\) and \(u_T\) are the order units of \(S\) and \(T\), then \(\phi\) is \emph{unital} if \(\phi(u_S) = u_T\). We denote the cone of completely positive maps from \(S\) to \(T\) by \({\rm CP}(S, T)\). The set of unital completely positive maps is denoted by \({\rm UCP}(S, T)\).

To simplify notation, we write
\[
\phi[a] \coloneqq (\phi \otimes {\rm id}_{{\rm Her}_s(\C)})(a)
\]
for \(a \in V \otimes {\rm Her}_s(\C)\). For \(b \in W \otimes {\rm Her}_t(\C)\), we write \(\phi^{-1}[b]\) for the preimage of \(B\) under \(\phi \otimes {\rm id}_{{\rm Her}_t(\C)}\). Likewise, \(\phi^{-1}[T]\) denotes the full level-wise preimage of the operator system \(T\).

For a given operator system \(S\) on $V$, its \emph{dual operator system} \(S^\vee\) is obtained by taking the dual cones level-wise. Since we are in the finite-dimensional setup, these will be proper convex cones and thus provide an operator system on the dual space $V'$.
Under the canonical isomorphism \[V' \otimes {\rm Her}_s(\C) \cong {\rm Lin}(V, {\rm Her}_s(\C))\] one has \(S^\vee_s = {\rm CP}(S, {\rm Mat}_s(\C))\). Biduality therefore gives the following separation criterion: \(a \in S_s\) if and only if $\phi(a)\geqslant 0$ for all \(\phi \in S^\vee_s\), if and only if \(\phi[a] \in {\rm Mat}_{s^2}(\C)_{+}\) for all \(\phi \in S^\vee_s\).

The Choi-Effros Theorem says that every abstract operator system is of the form $\phi^{-1}[A]$ for some $C^*$-algebra $A$ and unital completely positive map $\phi$.
It is therefore natural to ask for the smallest such  \(C^{*}\)-algebra. This is formalized by the \(C^{*}\)-envelope, whose existence was proven in \cite{hamana}. We recall the relevant definitions.

\begin{definition}
    Let \(S\) and \(T\) be operator systems. 

    ($i$) A map \(\phi \in {\rm UCP}(S,T)\) is called a \emph{realization of \(S\) on \(T\)} if \(\phi^{-1}[T] = S\).  In this case \((T, \phi)\) is also called an \emph{extension} of \(S\). If $T=A$ is a \(C^*\)-algebra, and \(\phi(S_1)\) generates \(A\) as a \(C^*\)-algebra, then  $(A,\phi$) is called  a \emph{\(C^{*}\)-extension} of \(S\).

    ($ii$) If \(T={\rm Mat}_d(A)\) for some commutative $C^*$-algebra $A$, we call $\phi$ a \emph{\(d\)-commutative realization}. For \(d=1\), we simply call it a \emph{commutative realization}. If \(A=\C\), we call it a \emph{\(d\)-dimensional realization}.

    ($iii$) 
    The \emph{\(C^{*}\)-envelope} of $S$ is the minimal \(C^{*}\)-extension of \(S\): it is a \(C^{*}\)-extension $$({C}^{*}_{\rm env}(S), \kappa)$$ with the following universal property. For every \(C^{*}\)-extension \((A, \phi)\) of \(S\), there is a surjective unital \(\ast\)-homomorphism \(\pi \colon  A \to {C}^{*}_{\rm env}(S)\) such that \(\kappa = \pi \circ \phi\).
\end{definition}

\begin{remark}\label{rem:fd}
   Both the class of finite-dimensional \(C^*\)-algebras and the class of commutative \(C^*\)-algebras is closed under quotients and \(C^*\)-subalgebras. Hence an operator system has a finite-dimensional or commutative realization if and only if its \(C^*\)-envelope is finite-dimensional or commutative, respectively.
\end{remark}

\begin{definition}\label{def:algsubhom}
    A \(C^*\)-algebra \(A\) is called \emph{\(d\)-subhomogeneous} if all of its irreducible representations have dimension at most \(d\). Equivalently, \(A\) is isomorphic to a \(C^*\)-subalgebra of \({\rm Mat}_d(B)\) for some commutative \(C^*\)-algebra \(B\). Further equivalent characterizations are that, for every \(C^*\)-algebra \(C\), every \(d\)-positive map \(A \to C\) is completely positive, and also that every \(d\)-positive map \(C \to A\) is completely positive; see \cite{bla,shu}.
\end{definition}

\begin{remark}\label{rem:dsubhom}
    The class of \(d\)-subhomogeneous \(C^*\)-algebras is also closed under quotients and \(C^*\)-subalgebras. Thus an operator system has a \(d\)-commutative realization if and only if its \(C^*\)-envelope is \(d\)-subhomogeneous.
\end{remark}

\section{Maximal Operator Systems}

Let \(C \subseteq V\) be a proper cone. There are, in general, many operator system structures on \(V\) whose first level is \(C\). Among them there is a smallest and a largest one with respect to level-wise inclusion\footnote{Some authors interchange this terminology, because the associated operator system norm varies in the opposite direction.}. The \emph{minimal operator system over \(C\)}, denoted \(C^{\rm min}\), is the smallest such structure, while the \emph{maximal operator system over \(C\)}, denoted \(C^{\rm max}\), is the largest. Explicit descriptions of these two constructions can be found, for instance, in \cite{fnt}.

The results in  \cite{xhabli} generalized these constructions by fixing not only the first level, but the first \(d\) levels of a given operator system. Thus one obtains the largest and smallest operator systems which agree with a prescribed operator system \(S\) up to level \(d\). We first record a useful characterization of the largest such structure.

\begin{proposition}\label{lem:dmax}
    Let \(S\) be an operator system on \(V\), and let \(d \in \N\) be fixed. For \(s \in \N\) and \(a \in V \otimes {\rm Her}_s(\C)\), the following are equivalent:
    \begin{itemize}
        \item[$(i)$] \(m^*am \in S_d\) for all \(m \in {\rm Mat}_{s,d}(\C)\).
        \item[$(ii)$] \(\phi[a] \in {\rm Mat}_{ds}(\C)_+\) for all \(\phi \in {\rm CP}(S,{\rm Mat}_d(\C))\).
        \item[$(iii)$] \(\phi[a] \in {\rm Mat}_{ds}(\C)_+\) for all \(\phi \in {\rm UCP}(S,{\rm Mat}_d(\C))\).
    \end{itemize}
    If \(s \leqslant d\), these conditions are further equivalent to \(a \in S_s\).
\end{proposition}

\begin{proof}
    The implication \((i) \Rightarrow (ii)\) is standard. Take \(v=\sum_{i=1}^d e_i \otimes v_i \in \C^d \otimes \C^s\), set \(m=(v_1,\ldots, v_d) \in {\rm Mat}_{s,d}(\C)\), and let \(e=\sum_{i=1}^d e_i \otimes e_i \in \C^d \otimes \C^d\). Then
    \[
        v^*\phi[a]v=e^*\phi[m^*am]e.
    \]
    Since \(\phi\) is completely positive and \(m^*am \in S_d\) by assumption, the right-hand side is nonnegative.

    Conversely, assume \((ii)\). By biduality, it suffices to show that \(\phi[m^*am] \in {\rm Mat}_{d^2}(\C)_+\) for every \(\phi \in {\rm CP}(S,{\rm Mat}_d(\C))\). This follows from
    \[
        \phi[m^*am]=(I_d\otimes m)^*\phi[a](I_d\otimes m),
    \]
    together with \((ii)\).

    The implication \((ii) \Rightarrow (iii)\) is immediate. For \((iii) \Rightarrow (ii)\), use the standard normalization argument: after splitting off zero blocks and conjugating by an invertible matrix, any completely positive map into \({\rm Mat}_d(\C)\) becomes unital. If the target size decreases in this process, we enlarge it again by taking a direct sum with normalized positive functionals.

    Finally, suppose \(s \leqslant d\). Then the preceding conditions imply \(a \in S_s\). Indeed, one may either apply \((ii)\) and biduality directly, or use \((i)\) to embed \(a\) into level \(d\) and then compress back to level \(s\).
\end{proof}

We can now state the definition from \cite{xhabli}. There, the term \emph{super \(d\)-maximal} is used; for readability we omit the prefix \emph{super}.

\begin{definition}
    Let \(S\) be an operator system on \(V\). The \emph{\(d\)-maximal operator system} \(S^{d-{\rm max}}\) over \(S\) is defined level-wise by the equivalent conditions in \Cref{lem:dmax}. It agrees with \(S\) up to level \(d\), and is the largest operator system with this property. We call \(S\) \emph{\(d\)-maximal} if \(S = S^{d-{\rm max}}\).
\end{definition}

\begin{remark}
    The \(d\)-maximal structures form a decreasing chain
    \[
        S^{\max}=S^{1-\max}\supseteq S^{2-\max} \supseteq \cdots\supseteq \bigcap_d S^{d-\max}=S.
    \]
    In particular, if \(S\) is \(d\)-maximal, then it is also \((d+1)\)-maximal.
\end{remark}

Following \cite{xhabli}, there is also a dual notion of \(d\)-minimality: one now takes the smallest operator system which agrees with \(S\) up to level \(d\). In the finite-dimensional setting considered here, this can be described concisely.

\begin{definition}\label{def:dmin}
    Let \(S\) be an operator system on \(V\). The \emph{\(d\)-minimal operator system} \(S^{d-{\rm min}}\) is the matrix conic hull of \(S_d\). Thus its cone at level \(s\) consists of all finite sums
    \[
        \sum_i v_i^* x_i v_i,
    \]
    where \(x_i \in S_d\) and \(v_i \in {\rm Mat}_{d,s}(\C)\).
\end{definition}

\begin{remark} \label{thm:automatic_CP}
    There is also a more intrinsic formulation of \(d\)-minimality and \(d\)-maximality for an operator system $S$, also see  \cite{xhabli, kiri2024}.
    Indeed, \(S\) is \(d\)-maximal if and only if every \(d\)-positive map \(\phi\colon T \to S\) from an arbitrary operator system \(T\) is completely positive.
    Similarly, \(S\) is \(d\)-minimal if and only if every \(d\)-positive map \(\phi\colon S \to T\) into an arbitrary operator system \(T\) is completely positive. In both cases one can restrict to operator systems $T$ on finite-dimensional spaces.
\end{remark}

\section{Subhomogeneous Operator Systems}

We now extend the notion of subhomogeneity from \(C^*\)-algebras to operator systems by means of several equivalent properties. The equivalence of \((i)\), \((iii)\), \((iv)\), as well as part of the connection with \((v)\) in the following theorem, already appears in \cite{xhabli}. Conditions ($iv$) and ($v$) are also studied in \cite{kiri2024}, where an affine version of ($v$) actually serves as definition of subhomogeneity.

\begin{theorem} \label{thm:subhom_OS}
    Let \(S\) be an operator system on a finite-dimensional vector space \(V\). The following are equivalent:

    \begin{enumerate}
    \item[\((i)\)] \(S\) has a \(d\)-commutative realization.
    \item[\((ii)\)] \(C^*_{\rm env}(S)\) is \(d\)-subhomogeneous.
    \item[\((iii)\)] \(S\) is \(d\)-maximal.
    \item[\((iv)\)] For every operator system \(T\), every \(d\)-positive map \(\phi\colon T \to S\) is completely positive.
    
    \item[\((v)\)] The dual system \(S^\vee\) is \(d\)-minimal.
    \item[\((vi)\)] For every \(s>d\) and every \(x \in \partial S_s\), there exists an  isometry \(v \in {\rm Mat}_{s,d}(\C)\) such that \(v^*xv \in \partial S_d\).
    \end{enumerate}
\end{theorem}

\begin{proof}
    By \Cref{rem:dsubhom}, items \((i)\) and \((ii)\) are equivalent. By \Cref{thm:automatic_CP}, items \((iii)\) and \((iv)\) are equivalent. 

    We now show \((i) \Rightarrow (iv)\). Let \(\psi\colon S\to {\rm Mat}_d(A)\) be a \(d\)-commutative realization of \(S\), and let $\phi\colon T\to S$ be $d$-positive. Then $\psi\circ\phi\colon T\to {\rm Mat}_d(A)$ is $d$-positive and thus completely positive, as explained in  \Cref{def:algsubhom}. Since $\psi$ is a realization, this implies that $\phi\colon T\to S$ is completely positive.

    We next prove ($iii$)$\Rightarrow$ ($i$), so  assume \(S=S^{d-{\rm max}}\). Let
    \[
        X={\rm UCP}(S,{\rm Mat}_d(\C)),
    \]
    with its natural compact Hausdorff topology; see \cite[Theorem 3.1]{xhabli}, here everything is even finite-dimensional. Define
    \[
        \psi\colon V \to {\rm Mat}_d({\rm C}(X)) \cong {\rm C}(X,{\rm Mat}_d(\C)), \qquad
        \psi(v)(\phi)=\phi(v).
    \]
    For \(a \in V \otimes {\rm Her}_t(\C)\), we have \(\psi[a](\phi)=\phi[a]\). Hence, by \Cref{lem:dmax},
    \[
        a \in S^{d-{\rm max}}_t
        \quad\Leftrightarrow\quad
        \psi[a] \in {\rm Mat}_{dt}({\rm C}(X))_+.
    \]
    Since \(S=S^{d-{\rm max}}\), this says exactly that  \(\psi\) is a \(d\)-commutative realization of \(S\).

    Next, \((iv) \Leftrightarrow (v)\) follows by duality. Indeed, assume \((iv)\), and let \(T\) be an operator system on a finite-dimensional space and \(\phi\colon S^\vee \to T\) a \(d\)-positive map. Its dual \(\phi'\colon T^\vee \to S\) is again \(d\)-positive, hence completely positive by \((iv)\). Dualizing again gives \(\phi \in {\rm CP}(S^\vee,T)\). By \Cref{thm:automatic_CP}, this is equivalent to \(S^\vee\) being \(d\)-minimal. The converse is the same argument in reverse.

    It remains to show \((iii) \Leftrightarrow (vi)\). 
    Assume first that \(S\) is \(d\)-maximal, and let \(x \in \partial S_s\).  If no compression \(v^*xv\), with \(v^*v=I_d\), lies in \(\partial S_d\), then all such compressions lie in \(\operatorname{int}(S_d)\). By compactness of the set of isometries, there is an \(\varepsilon>0\) such that
    \[
         v^*xv-\varepsilon u_d=v^*(x-\varepsilon u_s)v \in S_d
    \]
    for every isometry \(v \in {\rm Mat}_{s,d}(\C)\). Here, $u_s=u_1\otimes I_s$ denotes  the order unit. For an arbitrary \(w \in {\rm Mat}_{s,d}(\C)\), take the polar decomposition  \(w=vp\), with \(v\) an isometry and \(p\) psd. Then
    \[
       w^*(x-\epsilon u_s)w=p^*v^*(x-\epsilon u_s)vp \in S_d.
    \]
    From $d$-maximality of $S$ we obtain \(x-\epsilon u_s \in S_s\), contradicting \(x \in \partial S_s\). Hence some isometry \(v\) satisfies \(v^*xv \in \partial S_d\).

    Conversely, assume \((vi)\). If \(x \notin S_s\), then for some \(\epsilon>0\) the element \(x+\epsilon u_s\) lies in \(\partial S_s\). By \((vi)\), there is an isometry \(v\) such that
    \[
        v^*(x+\epsilon u_s)v=v^*xv+\epsilon u_d \in \partial S_d.
    \]
    This implies \(v^*xv \notin S_d\). So $S$ is indeed $d$-maximal.
\end{proof}

\begin{definition}
    We call an operator system $S$ {\it $d$-subhomogeneous}, if it fulfills the equivalent conditions from \Cref{thm:subhom_OS}.
\end{definition}

\begin{remark}
    \Cref{thm:subhom_OS} has several immediate consequences. First, the notion of $d$-subhomogeneity for operator system generalizes the one for $C^*$-algebras. Second, \(d\)-subhomogeneity passes from extension to subsystems: if \(T \hookrightarrow S\) is a realization and \(\psi\) is a \(d\)-commutative realization of \(S\), then \(\restr{\psi}{T}\) is a \(d\)-commutative realization of \(T\). Finally, item \((iv)\) is the analogue of automatic complete positivity for maps \emph{into} a \(d\)-subhomogeneous \(C^*\)-algebra. In contrast to $C^*$-algebras, \(d\)-positive maps \emph{out of} a $d$-subhomogeneous operator system \(S\) need not be completely positive in general.
\end{remark}

\begin{remark}
    For \emph{finite-dimensional realizations} of operator systems, there exist analogs of several items in \Cref{thm:subhom_OS} above. Actually ($ii$) becomes finite-dimensionality of the envelope (by \Cref{rem:fd}), ($v$) states that $S^\vee$ is generated by a {\it single} element from $S_d^\vee$ (see for example \cite[Theorem 2.9]{bergernetzer}), and ($vi$) has a stronger and more technical formulation in \cite[Theorem 2.3]{fnt}.
\end{remark}

The following corollaries are immediate.

\begin{corollary}
\label{cor:max}
    Let \(S\) be an operator system. The following are equivalent:
    \begin{itemize}
        \item[$(i)$] \(S\) has a commutative realization.
        \item[$(ii)$] \(C^*_{\rm env}(S)\) is commutative.
        \item[$(iii)$] \(S\) is the maximal operator system over \(S_1\).
    \end{itemize}
\end{corollary}

\begin{corollary}
    If an operator system has a \(d\)-dimensional realization, then it is \(d\)-maximal.
\end{corollary}

\begin{corollary}
\label{cor:comb}
    For an operator system \(S\), the following are equivalent:
    \begin{itemize}
        \item[$(i)$] \(S\) has both a finite-dimensional realization and a commutative realization, possibly different ones.
        \item[$(ii)$] \(C^*_{\rm env}(S)\) is finite dimensional and commutative.
        \item[$(iii)$] \(S_1\) is polyhedral and \(S\) is the maximal operator system over \(S_1\).
    \end{itemize}
\end{corollary}
\begin{proof}
    The equivalence of \((i)\) and \((ii)\) follows from \Cref{thm:subhom_OS}. The equivalence with \((iii)\) follows from \cite{fnt}: maximal operator systems over a cone have a finite-dimensional realization if and only if the cone is polyhedral, and in that case the realization can be chosen commutative.
\end{proof}

\section{Duals of Subhomogeneous Operator Systems}

\Cref{thm:subhom_OS} shows, among other things, that the dual of a \(d\)-subhomo\-geneous operator system is \(d\)-minimal, and thus will in general not be subhomogeneous. However, we can prove that duals of subhomogeneous systems are {\it projections/quotients} of subhomogeneous systems. This complements the same result for systems with a finite-dimensional realization, obtained in \cite{heltonklepmccullough}.

\begin{theorem}
    Let \(S\) be a \(d\)-subhomogeneous operator system on a finite-dimensional vector space \(V\). Then \(S^\vee\) is the quotient (i.e.\ the closed linear image) of a \(d\)-subhomogeneous operator system.
\end{theorem}

\begin{proof}
    Since \(S\) is \(d\)-subhomogeneous, we have a \(d\)-commutative realization \[\psi\colon S \hookrightarrow \mathrm{Mat}_d(A).\]  We will  show that \(\mathrm{Mat}_d(A)\) itself  projects onto \(S^\vee\).

    In the proof of \Cref{thm:subhom_OS} we have seen that we can choose $A={\rm C}(X)$ where $X={\rm UCP}(S,{\rm Mat}_d(\C))$ is a compact convex subset of a Euclidean space. 
    For \(x \in X\) and \(\operatorname{ev}_x\) the evaluation map at \(x\), we consider the map \[\pi_x\coloneqq (\operatorname{ev}_x \circ \psi)'\colon \mathrm{Mat}_d(\C) \to V'\] where we have identified ${\rm Mat}_d(\C)$ with its dual space, via the usual trace inner product. 

    Let $\lambda$ be any finite Borel measure with full support on $X$ (take for example the Lebesgue measure). We use it to define the following map: 
    $$\pi\colon \mathrm{Mat}_d({\rm C}(X)) \to V';\quad a \mapsto \int_X \pi_x(a(x)) {\rm d} \lambda(x).$$
    We first check that this is well-defined. For $a\in{\rm Mat}_d({\rm C}(X))$ and $x\in X$ we have 
    \[
        \pi_x(a(x)) = [v\mapsto {\rm tr}(a(x)^*\psi(v)(x))],
    \] which depends continuously on $x$. So the integrand is continuous on \(X\) and hence the integral exists. Thus $\pi$ is a well-defined $*$-linear map.

    We now check that \(\pi\) is completely positive as a map to \(S^\vee\). For $$a\in\left({\rm Mat}_d({\rm C}(X))\otimes{\rm Mat}_s(\C)\right)_+={\rm Mat}_{ds}({\rm C}(X))_+$$ and $b\in S_s$ we have $$\pi[a](b)=\int_X {\rm tr}(\underbrace{a(x)}_{\rm psd}\cdot\underbrace{\psi[b](x)}_{\rm psd}){\rm d}\lambda(x)\geqslant 0,$$ since $\psi$ is completely positive. This  implies $\pi[a]\in S^\vee_s,$ what was to be shown.

    We finally need to show that \(\pi\) is surjective onto \(S^\vee\), up to closure.  If $c\in V'\otimes{\rm Her}_s(\C)$ does not belong  to the closure of $\pi[{\rm Mat}_{ds}({\rm C}(X))_+],$ there exists some $b\in V\otimes{\rm Her}_s(\C)$ with $$c(b)< 0\ \mbox { and }\ \pi[a](b)\geqslant 0 \mbox{ for all } a\in {\rm Mat}_{ds}({\rm C}(X)_+.$$
    The second condition implies that $\psi[b]$ is positive, which is easily shown using that $\lambda$ has full support on $X$, and classical self-duality of the psd cone. Since  $\psi$ is a realization of $S,$ this further implies $b\in S_s$. This finally yields $c\notin S_s^\vee,$ which  finishes the proof.
\end{proof}

\begin{remark}
    In the case that $S$ has a $d$-dimensional realization, the above proof shows that $S^\vee$ is a projection of ${\rm Mat}_d(\C),$ thus generalizing the result from \cite{heltonklepmccullough}.
\end{remark}

\section{Examples}
Let \(Q \subseteq \R^{n+1}\) be the cone over the \(n\)-cube, i.e.\ the polyhedral cone 
defined by the \(2n\) inequalities $$ x_0 \pm x_i\geqslant 0,\quad i=1,\ldots n.$$ By \cite{fnt}, the maximal operator system over $Q$ is obtained by imposing the same inequalities at every level: For $a=(a_0,\ldots, a_n)\in{\rm Her}_s(\C)^{n+1}$ we have
$$a \in Q^{\rm max}_s \ \Leftrightarrow\ a_0\pm a_i\in  {\rm Mat}_{s}(\C)_+$$
for \(i=1,\ldots, n\). We now ask whether this system is $d$-minimal for some $d$. By \Cref{thm:subhom_OS} this is equivalent to its dual system being $d$-subhomogeneous. The dual system is the minimal system over the dual cone, see \cite{bergernetzer} for example. The dual cone of $Q$ is the cone over the cross-polytope, which we denote by $C$. It has $2^n$ facets and $2n$ extreme rays. For $n\leqslant 2$ the cones $Q$ and $C$ are isomorphic, but for $n\geqslant 3$ they are different.

Now  it turns out that $d$-minimality and $d$-subhomogeneity depend sharply on the dimension. Note that very similar arguments can also be found in  \cite{farenickmalekivarelasingla,krielmatrixconvex} in the context of matrix extreme points of free spectrahedra.

\begin{theorem} \label{thm:cube_d_min}
    Let \(Q\subseteq \R^{n+1}\) be the cone over the cube and $C\subseteq\R^{n+1}$ the cone over the cross-polytope.
    \begin{enumerate}
        \item[\((i)\)] For \(n=1\), \(Q^\mathrm{max}\) is $1$-minimal and $C^{\rm min}$ is $1$-subhomogeneous.
        \item[\((ii)\)] For \(n=2\), \(Q^\mathrm{max}\) is $2$-minimal and $C^{\rm min}$ is $2$-subhomogeneous.
        \item[\((iii)\)] For \(n \geqslant 3\), \(Q^\mathrm{max}\) is not \(d\)-minimal  and $C^{\rm min}$ is not $d$-subhomogeneous for any \(d \in \N.\)
    \end{enumerate}
\end{theorem}

\begin{proof}
    For $n=1$, \(Q\) is a simplex cone, and thus minimal and maximal systems coincide (see for example \cite{fnt}). For \(n=2\) we need to show that $Q_2^{\max}$ generates all higher levels of $Q^{\max}$, as explained in \Cref{def:dmin}. So let \(s \geqslant 3\) and  take \(a \in Q^{\rm max}_s\). Since $a_0\pm a_i\geqslant 0$ holds for $i=1,2,$ we must have $a_0\geqslant 0.$ Further, any congruence transformation that splits off a zero block in $a_0$ simultaneously splits off the same zero block in $a_1$ and $a_2$. After potentially decreasing $s$ we can thus assume $a_0=I_s$. The inequalities now state that the eigenvalues of $a_1,a_2$ are in $[-1,1].$ Since  classical convex combinations are allowed in the construction from \Cref{def:dmin}, we can even assume that these eigenvalues are all $\pm 1$, and thus $$a_i=2p_i-I_s$$ for some orthogonal projections $p_1,p_2.$ By Halmos' two projections theorem, $p_1,p_2$ (and thus also $a_1,a_2$) can be decomposed simultaneously into blocks of size at most $2$, by a unitary conjugation. Since block-sums are allowed in the construction from \Cref{def:dmin}, we are done.

    For \(n\geqslant 3\), fix \(s\in\N\)  and choose orthogonal projections \(p_1, \dots, p_n\) which generate \({\rm Mat}_s(\C)\) as an algebra; this is possible by \cite{davisgenerators}. Set \(a_i=2p_i-I_s\) and consider $$a=(I_s, a_1, \dots, a_n)\in Q_s^{\max}.$$  We will show that $a$ does not arise from lower levels of $Q^{\max}$ as in \Cref{def:dmin}. But we first show that $a$ lies on a classical extreme ray of $Q_s^{\max}.$ 
    
    By the Krein-Milman theorem we can write \(a = \sum_j b^{(j)}\) for \(b^{(j)}\) on extreme rays of $Q_s^{\max}$. We thus have \[\sum_j (b_0^{(j)} \pm b_i^{(j)}) = I_s \pm a_i\] for each \(i=1,\ldots n\), and by the inequalities defining $Q^{\max}$ each term on the left-hand side is positive semidefinite. Hence, we get $$\ker(I_s \pm a_i) \subseteq \ker(b_0^{(j)} \pm b_i^{(j)})$$ for all \(i=1,\ldots,n\). 
    
    This implies that each  \(b_0^{(j)}\) is proportional to the identity: taking \(v \in \ker(p_i)=\ker(I_s + a_i) \) and \(w \in \ker(p_i)^\perp=\ker(I_s - p_i)=\ker(I_s - a_i)\), we get \[w^* (b_0^{(j)} + b_i^{(j)}) v = 0 = v^* (b_0^{(j)} - b_i^{(j)}) w.\]  Taking the hermitian conjugate of the second term and adding both yields \(w^* b_0^{(j)} v = 0\). This implies that $\ker(p_i)$ is an invariant subspace of \(b_0^{(j)}\), and thus $p_i$ and $b_0^{(j)}$ commute. But since the \(p_i\) generate \({\rm Mat}_s(\C)\),  \(b_0^{(j)}\) is indeed proportional to $I_s,$ say $b_0^{(j)}=\alpha_jI_s.$ 
    
    Since $b^{(j)}$ is classically extreme in $Q_s^{\max}$, the matrices  \(\alpha_j^{-1} b^{(j)}_i\) for $i=1,\ldots, n$ have eigenvalues \({\pm 1}\), so we write them as $$\alpha_j^{-1} b^{(j)}_i=2q_i^{(j)} - I_s$$ for projections \(q_i^{(j)}\). Looking at the sum above,  we get \(1 = \sum_j \alpha_j\) and \(p_i = \sum_j \alpha_j q_i^{(j)}\) for each \(i=1,\ldots, n\). Now it is straightforward to show that this can only be true if all $q_i^{(j)}$ equal $p_i$. We have thus proven that $a$ is indeed extreme in $Q_s^{\max}$.

    Now finally assume that $a$ arises from a lower level $d<s$ as in \Cref{def:dmin}:
    $$a=\sum_j v_j^* x_jv_j$$ with $x_j\in Q_d^{\max}.$
    By extremality of $a$, each term $v_j^*x_jv_j$ is already proportional to $a$. But the first entry of $a$ is $I_s$, the first entry of $v_j^*x_jv_j$ can have at most rank $d$. This is a contradiction. So $Q^{\max}$ is indeed not $d$-minimal.
\end{proof}

\begin{example}
\label{ex:sep}
    Let \(P={\rm Mat}_2(\C)_+\). Then the minimal operator system \(P^{\min}\) is not \(3\)-subhomogeneous.
    Indeed, at level \(s\), the cone \(P_s^{\min}\) is the cone of separable positive semidefinite elements in \({\rm Mat}_2(\C)\otimes {\rm Mat}_s(\C)\). If \(P^{\min}\) were \(3\)-subhomogeneous, then it would be \(3\)-maximal by \Cref{thm:subhom_OS}. Hence, by \Cref{lem:dmax}, every entangled state in \({\rm Mat}_2(\C)\otimes {\rm Mat}_s(\C)\) would have an entangled compression of the second tensor factor to dimension \(3\).

    However, there exist PPT entangled states in \({\rm Mat}_2(\C)\otimes {\rm Mat}_4(\C)\); see \cite{horodecki1997}. Let \(\sigma\) be such a state. For every isometry \(v\in{\rm Mat}_{4,3}(\C)\), the compression
    \[
        (I_2\otimes v)^*\sigma(I_2\otimes v)
    \]
    is again PPT, because partial transposition on the first tensor factor commutes with compression on the second tensor factor. By the Peres--Horodecki criterion, PPT is equivalent to separability in dimensions \(2\times 3\); see \cite{horodecki1996}. Thus every such \(3\)-dimensional compression of \(\sigma\) is separable. This contradicts the compression criterion above, and therefore \(P^{\min}\) is not \(3\)-subhomogeneous.
\end{example}

\begin{remark}
    Clearly, $d$-subhomogeneity passes from extensions to subsystems and $d$-minimality passes to quotients/linear images. The other directions fail in general. For example, \cite{Drescher} provides a $8$-dimensional concrete operator system $S\subseteq{\rm Mat}_3(\C)$ which is not $d$-minimal for any $d.$ Note that ${\rm Mat}_3(\C)$ is $3$-minimal. The dual $S^\vee$ is thus not $d$-subhomogeneous for any $d$, although it is a quotient of ${\rm Mat}_3(\C),$ which is $3$-subhomogeneous.
\end{remark}

\section*{Some Open Questions}

We close with several open questions suggested by the results above.

Let $C\subseteq\R^{n+1}$ be a polyhedral cone. What can be said, in general, about subhomogeneity of $C^{\min}$? For simplex cones the question is trivial, while the first nontrivial case in dimension $n=2$ is $2$-subhomogeneous. Is this true for every polyhedral cone in $\R^3$? Does the answer depend on the number and arrangement of its extreme rays?

In dimensions $n\geqslant 3$, our examples provide counterexamples to subhomogeneity. These examples have $2n$ extreme rays, whereas simplex cones have only $n+1$ extreme rays. It is therefore natural to ask what happens for polyhedral cones whose number of extreme rays lies between these two cases.

One may also ask analogous questions for minimal systems over non-polyhedral cones. A particularly interesting case, because of its connection with quantum information, is the cone of positive semidefinite matrices $P={\rm Mat}_n(\C)_+$, see also \Cref{ex:sep}. For the minimal system over $P$ one has
\[
P_s^{\min}={\rm Mat}_n(\C)_+\otimes_{\min}{\rm Mat}_s(\C)_+,
\]
the cone of mixed separable states. Thus $d$-subhomogeneity would mean that entanglement can always be detected after compressing the second tensor factor to dimension $d$. The preceding example shows that for $n=2$ one cannot take $d=3$, but it does not rule out the possibility of some larger $d$.

Equivalently, one can ask the following question:

Given $n\geqslant 1$, does there exist $d\geqslant 1$ such that, for every $s\geqslant 1$ and every mixed entangled state $\sigma\in{\rm Mat}_n(\C)\otimes{\rm Mat}_s$, some compression of the second tensor factor to ${\rm Mat}_d(\C)$ remains entangled?

\section*{Acknowledgments} 
The authors thank Andreas Thom for illuminating discussions on subhomogeneous $C^*$-algebras.

This research was funded in part by the Austrian Science Fund (FWF) [10.55776/P36684]. For open access purposes, the authors have applied a CC BY public copyright license to any author-accepted manuscript version arising from this submission. 

\newpage\thispagestyle{empty}
\addcontentsline{toc}{part}{Bibliography} 
{\linespread{1}\bibliographystyle{dpbib} \bibliography{references}}

\end{document}